\documentclass{amsart}
\usepackage{amsmath, amscd, amssymb, amsthm}
\usepackage{bbm}
\usepackage{latexsym}
\usepackage{amsfonts}
\usepackage{graphicx}
\usepackage{xcolor}

\usepackage[all,cmtip]{xy}
\usepackage[colorlinks,linkcolor=blue,breaklinks=true,urlcolor=blue,citecolor=blue,anchorcolor=blue,pagebackref]{hyperref}%
\usepackage{geometry}
\newtheorem{theorem}{Theorem}
\newtheorem{lemma}{Lemma}
\newtheorem{corollary}[theorem]{Corollary}

\newtheorem{proposition}{Proposition}

\newtheorem{conjecture}{Conjecture}

\renewcommand*\backref[1]{}
\renewcommand*\backrefalt[4]{ \ifcase #1 \or (cited on page #2) \else (cited on pages #2) \fi}

\newcommand{\be}{\begin{equation}}
\newcommand{\ee}{\end{equation}}
\newcommand{\bea}{\begin{eqnarray}}
\newcommand{\eea}{\end{eqnarray}}

\def\XXint#1#2#3{{\setbox0=\hbox{$#1{#2#3}{\int}$ }
\vcenter{\hbox{$#2#3$ }}\kern-.6\wd0}}

\begin{document}

\title[Balanced Bismut torsion-parallel fourfolds with constant holomorphic sectional curvature]{Balanced Bismut torsion-parallel fourfolds with constant holomorphic sectional curvature}

\author{Qingsong Wang}
\address{Qingsong Wang. Hal{\i}c{\i}o\u{g}lu Data Science Institute, University of California San Diego, La Jolla, CA 92093, USA}
\email{qswang92@gmail.com}
\thanks{Zheng is the corresponding author. He is partially supported by the National Natural Science Foundation of China
under Grant Nos.~12471039 and~12141101, and is supported by the 111 Project D21024.}

\author{Fangyang Zheng}
\address{Fangyang Zheng. School of Mathematical Sciences, Chongqing Normal University, Chongqing 401331, China}
\email{20190045@cqnu.edu.cn; franciszheng@yahoo.com}


\begin{abstract}
A long-standing conjecture in non-K\"ahler geometry asserts that a compact Hermitian manifold with constant Chern holomorphic sectional curvature is K\"ahler if the constant is non-zero and Chern flat if it is zero. In complex dimension $2$, the conjecture was established by Balas and Gauduchon in 1985 for non-positive curvature and by Apostolov, Davidov, and Muskarov in 1996 in general. In higher dimensions, known cases include twistor spaces, by work of Davidov, Grantcharov, and Muskarov; locally conformally K\"ahler manifolds with non-positive constant curvature, by work of H. Chen, L. Chen, and Nie; and non-balanced Bismut torsion parallel (BTP) manifolds, by work of S. Chen and Zheng. Using a classification result of Zhao and Zheng, S. Chen and Zheng also established the conjecture for balanced BTP threefolds. In this article, we prove the conjecture for all compact balanced BTP fourfolds. Although balanced BTP manifolds are highly restrictive, a classification remains unavailable in complex dimensions $4$ and higher. Our analysis may provide insight into their structure in complex dimension $4$.
\end{abstract}

\maketitle

\tableofcontents

\markleft{Qingsong Wang and Fangyang Zheng}
\markright{BTP metrics with constant holomorphic sectional curvature}

\section{Introduction and statement of results}\label{intro}

A long-standing conjecture in non-K\"ahler geometry concerns the classification of ``Hermitian space forms'':

\begin{conjecture}[{\bf Constant Holomorphic Sectional Curvature Conjecture}] \label{conj1}
Any compact Hermitian manifold with constant holomorphic sectional curvature must be either K\"ahler or Chern flat.
\end{conjecture}

For a given Hermitian manifold $(M^n,g)$, the holomorphic sectional curvature in the direction of $X$ is  defined by
$ H(X) = R_{X\bar{X}X\bar{X}}/|X|^4$,
where $X$ is any non-zero complex tangent vector of $(1,0)$ type and $R$ is the curvature tensor of the Chern connection $\nabla$ of $g$.

Recall that complete K\"ahler manifolds with constant holomorphic sectional curvature are called {\em complex space forms}. Their universal covers are the complex projective space ${\mathbb C}{\mathbb P}^n$, the complex Euclidean space ${\mathbb C}^n$, or the complex hyperbolic space ${\mathbb C}{\mathbb H}^n$, equipped with (a constant multiple of) the standard metrics. On the other hand, by the classical result of Boothby \cite{Boothby} in 1958, compact Chern flat manifolds are exactly compact quotients of complex Lie groups (equipped with left-invariant Hermitian metrics). Compact Chern flat manifolds can be non-K\"ahler, and often are, when $n\geq 3$ (while in the non-compact case there exist  non-K\"ahler Chern flat complete Hermitian surfaces \cite{YangZ}).

So the above conjecture simply says that if $(M^n,g)$ is a compact Hermitian manifold with $H=c$, then either $g$ is K\"ahler (hence $(M^n,g)$ is a complex space form), or $g$ is Chern flat (namely, $R=0$), in which case its universal cover is a complex Lie group (equipped with a left-invariant Hermitian metric).

Conjecture \ref{conj1} is known to be true when $n=2$, by the work of Balas and Gauduchon \cite{BG} (see also \cite{Balas}) in 1985 for the $c\leq 0$ case, and by
Apostolov, Davidov, and Muskarov \cite{ADM} in 1996 in the general case, as a corollary of their classification theorem for compact self-dual Hermitian surfaces.

For $n\geq 3$, the conjecture is still largely open, except in some special cases. The first substantial result towards the conjecture is the one obtained by Davidov, Grantcharov, and Muskarov \cite{DGM}, in which they showed among other things that the only twistor space with constant holomorphic sectional curvature is the complex space form ${\mathbb C}{\mathbb P}^3$. More recently, H. Chen, L. Chen, and Nie \cite{CCN} considered locally conformally K\"ahler manifolds and confirmed the conjecture in the $c\leq 0$ cases. They also pointed out the necessity of the compactness assumption in the conjecture by explicit examples. Tang \cite{Tang} proved the conjecture under the additional assumption that the metric is Chern K\"ahler-like, meaning that the curvature tensor $R$ obeys all the K\"ahler symmetries. In \cite{ZhouZ}, Zhou and Zheng proved that any compact balanced threefold with zero {\em real bisectional curvature}, a notion introduced in \cite{XYangZ} which is slightly stronger than $H$, must be Chern flat. Also, in \cite{LZ} and \cite{RZ}, Zheng and collaborators confirmed the conjecture under the additional assumption that either $(M^n,g)$ is a complex nilmanifold with nilpotent complex structure $J$ (in the sense of \cite{CFGU}), or $g$ is Bismut K\"ahler-like (BKL), meaning that the curvature $R^b$ of the Bismut connection $\nabla^b$ of $g$ (\cite{Bismut}) obeys all K\"ahler symmetries.

A Hermitian metric is called {\em Bismut torsion-parallel} (BTP) if $\nabla^bT^b=0$, where $T^b$ is the torsion tensor of the Bismut connection $\nabla^b$. All non-K\"ahler BKL manifolds and all Vaisman manifolds are examples of non-balanced BTP manifolds, by the work of Zhao and Zheng \cite{ZZCrelle} and   by the work of Andrada and Villacampa \cite{AndradaV}, respectively.

Recall that a Hermitian metric $g$ is said to be {\em balanced} if $d(\omega^{n-1})=0$, where $\omega$ is the K\"ahler form of $g$ and $n$ is the complex dimension of the manifold. BTP metrics are systematically studied in \cite{ZhaoZ24}, where it was shown that the BTP condition is equivalent to some partial symmetry conditions on $R^b$, generalizing the AOUV conjecture which states that BKL $\subset$ BTP. The conjecture was raised by Angella, Otal, Ugarte, and Villacampa in \cite{AOUV} and was confirmed in \cite{ZZCrelle}.

The majority of BTP manifolds are non-balanced, but there are also balanced BTP manifolds when the complex dimension $n\geq 3$, though such metrics tend to be highly restrictive. In their recent work \cite{ZhaoZ25}, Zhao and Zheng give a coarse classification result for compact balanced (but non-K\"ahler) BTP threefolds.

In \cite[Theorem 1.2]{ChenZ26}, S. Chen and Zheng confirmed Conjecture \ref{conj1} for non-balanced BTP manifolds in all dimensions, and also for balanced BTP threefolds, utilizing the classification results of \cite{ZhaoZ25}. The main purpose of this article is to confirm Conjecture \ref{conj1} for all compact balanced BTP fourfolds:

\begin{theorem} \label{thm1}
Let $(M^4,g)$ be a compact Hermitian manifold of complex dimension $4$ such that $g$ is balanced and Bismut torsion-parallel (BTP). If $g$ has constant Chern holomorphic sectional curvature, then it must be either K\"ahler or Chern flat.
\end{theorem}

We remark that for BTP manifolds, when the metric is non-balanced, the vector field $X$ dual to Gauduchon's torsion $1$-form $\eta$ (\cite{Gauduchon}) is a holomorphic vector field with constant norm, and $X$ belongs to the kernel of the Bismut curvature tensor $R^b$. This gives a starting point in the study of Conjecture \ref{conj1} and makes the analysis relatively easy. In the balanced (but non-K\"ahler) case, on the other hand, $\eta =0$ and we do not have a flat direction of $R^b$ to work with, so the study of Conjecture \ref{conj1} becomes more challenging. As observed in \cite{ZhouZ}, balanced threefolds admit special unitary frames under which the torsion tensor takes  a particularly simple form. This is the starting point in the coarse classification for balanced BTP threefolds in \cite{ZhaoZ25}, which enabled Chen and Zheng to prove Conjecture \ref{conj1} for such threefolds. In dimensions $n\geq 4$, this is no longer the case. In what follows we will give an outline of our strategy for proving Theorem \ref{thm1}.

Let $(M^n,g)$ be a BTP manifold. Let $e=\{ e_1, \ldots , e_n\}$ be a local unitary frame of type $(1,0)$ tangent vectors, and denote by  $\varphi =\{ \varphi_1, \ldots , \varphi_n\}$ the coframe of local $(1,0)$-forms dual to $e$ (namely, $\varphi_i(e_j)=\delta_{ij}$ for any $1\leq i,j\leq n$). Let $T^j_{ik}$ be the components under $e$ of the Chern torsion tensor $T$, and $R^b$ the curvature tensor of the Bismut connection $\nabla^b$. Note that the Bismut torsion $T^b$ can be expressed in terms of the Chern torsion $T$ and vice versa, so the BTP condition $\nabla^bT^b=0$ is equivalent to $\nabla^bT=0$ (\cite{ZhaoZ24}). Since $\nabla^bT=0$ implies $(R^b\cdot T)^{\ell}_{pq}=0$, the following identity holds on any BTP manifold under any local unitary frame $e$:
\begin{equation}  \label{eq:stab}
 (R^b_{i\bar{j}}\cdot T)^{\ell}_{pq} := \sum_{r=1}^n \big\{ R^b_{i\bar{j}r\bar{\ell}}T^r_{pq} - R^b_{i\bar{j}p\bar{r}}T^{\ell}_{rq} - R^b_{i\bar{j}q\bar{r}}T^{\ell}_{pr} \big\} =0, \ \ \ \ \ \ \ \ \ \forall \ 1\leq i,j,\ell , p, q \leq n.
 \end{equation}
Next, let us consider the expression
\begin{equation}
\sigma_B=\sqrt{-1}\sum_{i,j=1}^n B_{i\bar{j}} \,\varphi_i \wedge \overline{\varphi}_j =\sqrt{-1}\sum_{i,j=1}^n \big( \sum_{r,s=1}^nT^j_{rs} \overline{T^i_{rs}} \big) \,\varphi_i \wedge \overline{\varphi}_j.
\end{equation}
Clearly, this is independent of the choice of the local unitary frame $e$ and hence is a globally defined non-negative $(1,1)$-form on $M^n$. We will denote by $r_B$ the number of positive eigenvalues of $\sigma_B$, and call it the {\em rank of the $B$ tensor}. Since $\nabla^bB=0$, $r_B$ is a constant.

As observed in \cite{ChenZ26}, under the BTP and constant Chern holomorphic sectional curvature assumptions, the Bismut curvature components $R^b_{i\bar{j}k\bar{\ell}}$ can be explicitly expressed as a quadratic polynomial in $T$. So (\ref{eq:stab}) becomes an over-determined system of cubic equations in $T$. When $n=4$, we will divide the proof  into two cases, depending on whether $r_B<4$ or $r_B=4$. The first case is easier as the torsion components $T^j_{ik}$ have more zero terms. In the second case, which leads to a contradiction and therefore cannot occur, we will focus our discussion on the identity obtained by contracting (\ref{eq:stab}) after simultaneously diagonalizing $B$ and $S=Ric^b$. When $S$ is not identically zero, this reduced form of (\ref{eq:stab}) will allow us to rule out the case $r_B=4$. This is the content of \S 4. The more technical case when $S$ is identically zero will be  discussed in the last section.

The proof, together with the classification result of \cite{ZhaoZ25}, also leads to the following explicit description of the Chern flat case:

\begin{corollary}
Let $(M^4,g)$ be a compact, balanced BTP manifold in complex dimension $4$ with constant Chern holomorphic sectional curvature. Then one of the following holds:
\begin{enumerate}
\item  $g$ is K\"ahler. In this case $(M^4,g)$ is a complex space form.
\item  $g$ is non-K\"ahler and Chern flat. In this case the universal cover of $(M^4,g)$ is holomorphically isometric to $(\mbox{SL}(2,{\mathbb C}), \lambda g_{\kappa}) \times ({\mathbb C},g_0)$. Here $g_0$ is the flat metric on ${\mathbb C}$, $\lambda >0$ is a constant, and $g_{\kappa}$ is the standard Killing metric.
    \end{enumerate}
\end{corollary}

Let $\{ X, Y, Z\}$ be a basis of the Lie algebra ${\mathfrak s}{\mathfrak l}(2,{\mathbb C})$ so that $[X,Y]=Z$, $[Y,Z]=X$, and $[Z,X]=Y$. Using $\{ X,Y,Z\}$ as a unitary basis defines a Hermitian inner product on ${\mathfrak s}{\mathfrak l}(2,{\mathbb C})$ which is independent of the choice of the basis. This inner product corresponds to a left-invariant Hermitian metric $g_{\kappa}$ on the simple complex Lie group $\mbox{SL}(2,{\mathbb C})$. The metric is balanced, BTP, and  Chern flat.

In general dimensions, Podest\`a and Zheng characterized compact BTP manifolds that are Chern flat \cite[Theorem 1.2, Theorem 3.4]{PodestaZ}. Their universal covers are complex reductive Lie groups. The BTP metrics preserve the product structure and are unique up to a Killing-isometry on each simple factor. In particular, in complex dimensions $3$ and $4$, up to a constant multiple of the metric, the only (non-K\"ahler) compact, Chern flat BTP manifolds are quotients of $(\mbox{SL}(2,{\mathbb C}), g_{\kappa})$ or $(\mbox{SL}(2,{\mathbb C}), g_{\kappa}) \times ({\mathbb C}, g_0)$.

We conclude this introduction with a few remarks. First, in Conjecture \ref{conj1}, one can replace the constant Chern holomorphic sectional curvature assumption by the constant Levi-Civita (or Bismut) holomorphic sectional curvature assumption. The proof of Theorem \ref{thm1} is very likely extendable to such cases. Secondly, a natural question is how to extend Theorem \ref{thm1} to higher dimensions, namely, to confirm Conjecture \ref{conj1} for balanced BTP manifolds in dimensions $5$ or higher. The proof of Theorem \ref{thm1} uses a detailed case-by-case analysis that does not seem to extend easily, and new ideas and insight are needed. Thirdly, balanced BTP manifolds form a highly restrictive class of special Hermitian manifolds, though its classification is still wide open in dimensions $4$ or higher. A technical difficulty is that we do not know where to start, as the $3$-dimensional case was indeed too special. Although the analysis in this article concerns the special case of constant Chern holomorphic sectional curvature, it provides information about the structure of balanced BTP fourfolds. Hopefully the discussions can be pushed further to generate some useful information which might help lead to some classification result for such manifolds in dimension $4$ or higher.

\vspace{0.3cm}

\section{Preliminaries}

We begin by fixing notation. Let $(M^n,g)$ be a Hermitian manifold. Denote the Chern and Bismut connections of $g$ by $\nabla$ and $\nabla^b$, their torsion tensors by $T$ and $T^b$, and their curvature tensors by $R$ and $R^b$, respectively. The Chern holomorphic sectional curvature is equal to a constant $c$ if and only if
$$ \widehat{ R}_{i\bar{j}k\bar{\ell}}  : = \frac{1}{4}\big( R_{i\bar{j}k\bar{\ell}} + R_{k\bar{j}i\bar{\ell}} + R_{i\bar{\ell}k\bar{j}} + R_{k\bar{\ell}i\bar{j}} \big) =  \frac{c}{2} \big( \delta_{ij}\delta_{k\ell} + \delta_{i\ell}\delta_{kj} \big) ,$$
for any $1\leq i,j,k,\ell \leq n$ under any local unitary frame $e$. When $g$ is BTP and has constant Chern holomorphic sectional curvature, \cite[Lemma 8]{ChenZ26} gives an expression for $R^b_{i\bar{j}k\bar{\ell}}$. Note that the Chern connection corresponds to the $t=0$ case in the family of Gauduchon connections, while the notation $v$ and $w$ was defined in the formula five lines above Lemma 7 in \cite{ChenZ26}.

\begin{lemma}[\cite{ChenZ26}]  \label{lemmaRb}
Let $(M^n,g)$ be a BTP manifold whose Chern holomorphic sectional curvature is equal to a constant $c$. Then under any local unitary frame $e$,
\begin{equation} \label{eq:Rb}
R^b_{i\bar{j}k\bar{\ell}} = \frac{c}{2} \big( \delta_{ij}\delta_{k\ell} + \delta_{i\ell}\delta_{kj} \big) - \frac{1}{2}\sum_{r=1}^n T^r_{ik} \overline{ T^r_{j\ell }} - \frac{3}{4}\sum_{r=1}^n \big( T^j_{ir} \overline{ T^k_{\ell r}} + T^{\ell}_{kr} \overline{ T^i_{j r}}  \big) + \frac{1}{4}\sum_{r=1}^n \big( T^{\ell}_{ir} \overline{ T^k_{j r}} + T^{j}_{kr} \overline{ T^i_{\ell r}}  \big),
\end{equation}
for any $1\leq i,j,k,\ell \leq n$.
\end{lemma}

Recall that Gauduchon's torsion $1$-form $\eta$ is defined by $d(\omega^{n-1})=-\eta\wedge \omega^{n-1}$ where $\omega$ is the K\"ahler form of $g$. By definition, the metric $g$ is balanced if $\eta =0$. Write $\eta=\sum_i\eta_i\varphi_i$, where $\varphi$ is the coframe dual to $e$. Then $\eta_i = \sum_{k=1}^n T^k_{ki}$ for each $i$ (see \cite{ZhaoZ24} for instance). Setting $k=\ell$ in (\ref{eq:Rb}) and summing over $k$ from $1$ to $n$ gives the following expression for the first Bismut Ricci tensor $S$:

\begin{lemma} \label{lemmaS}
Let $(M^n,g)$ be a balanced BTP manifold with constant Chern holomorphic sectional curvature $c$. Then under any local unitary frame $e$, the first Bismut Ricci curvature tensor $S$ has components
\begin{equation} \label{eq:S}
S_{i\bar{j}} : = \sum_{k=1}^n R^b_{i\bar{j}k\bar{k}} = \frac{c}{2}(n+1) \delta_{ij} + \frac{1}{4}\big( B_{i\bar{j}} - A_{i\bar{j}}  \big), \ \ \ \ \ \ \ \forall \ 1\leq i,j\leq n,
\end{equation}
where $A_{i\bar{j}} = \sum_{r,s}T^r_{is} \overline{ T^r_{js}} $, $B_{i\bar{j}} = \sum_{r,s}T^j_{rs} \overline{ T^i_{rs}} $.
\end{lemma}

Note that both $A$ and $B$ are independent of the choice of the local unitary frame $e$ and thus are globally defined $(1,1)$ tensors on $M^n$. Their traces are clearly equal, so the trace of $S$ is given by
\begin{equation} \label{eq:traceS}
\mbox{tr}(S) := \sum_{i=1}^n S_{i\bar{i}} \,= \,\frac{c}{2}n(n+1).
\end{equation}

Fix $p\in M$, and let $V=T_pM\cong{\mathbb C}^n$ denote the holomorphic tangent space at $p$. Denote by $\mbox{ker}(B)=\{ v \in V \mid B_{v\bar{u}}=0 \ \,\forall \, u\in V\}$ the kernel of $B$, and by $(\mbox{ker}(B))^{\perp}$ its orthogonal complement in $V$. We have the following lemma:
\begin{lemma} \label{lemmakernelB}
Let $(M^n,g)$ be a BTP manifold. Then, under any local unitary frame $e$, we have
\begin{equation} \label{eq:kernelB}
 R^b_{i\bar{j}k\bar{\ell}} = 0,  \ \ \ \ \ \ \ \mbox{if} \ e_k\in (\mbox{ker}(B))^{\perp} \  \mbox{and} \ \,e_{\ell} \in \mbox{ker}(B),
\end{equation}
for any $1\leq i,j\leq n$.
\end{lemma}

\begin{proof}
Let $e$ be a local unitary frame. Under the BTP assumption, the equation (\ref{eq:stab}) holds. Now if $e_k\in (\mbox{ker}(B))^{\perp}$,  $e_{\ell} \in \mbox{ker}(B)$, then by the definition of $B$ we have $T^{\ell}_{pq}=0$ for any $1\leq p,q\leq n$. So (\ref{eq:stab}) becomes
$$ \sum_{r=1}^n  R^b_{i\bar{j}r\bar{\ell}}\,T^r_{pq}  =0, \ \ \ \ \ \ \ \ \ \ \forall \ 1\leq i,j,p,q\leq n. $$
Multiplying both sides of the above equation by $\overline{T^k_{pq}}$ and summing over $p$ and $q$, we obtain $\sum_r R^b_{i\bar{j}r\bar{\ell}}\,B_{k\bar{r}} = 0$. Since $B$ is positive definite on $(\mbox{ker}(B))^{\perp}$, this means that $R^b_{i\bar{j}k\bar{\ell}}=0$ whenever $e_k\in (\mbox{ker}(B))^{\perp}$. This finishes the proof of the lemma.
\end{proof}

Next we show that under the BTP condition, the tensors $B$ and $S$ can be simultaneously diagonalized:

\begin{lemma} \label{lemmadiag}
Let $(M^n,g)$ be a BTP manifold. Then locally there always exist  unitary frames under which both $B$ and $S$ are diagonal, namely,  $S_{i\bar{j}}=B_{i\bar{j}}=0$ for all  $1\leq i\neq j\leq n$.
\end{lemma}

\begin{proof}
Let $e$ be any local unitary frame. Under the BTP assumption, the Bismut curvature obeys the symmetry condition $R^b_{i\bar{j}k\bar{\ell}} = R^b_{k\bar{\ell}i\bar{j}}$ for any $1\leq i,j,k,\ell \leq n$. Denote the connection and curvature matrices of $\nabla^b$ under $e$ by $\theta^b$ and $\Theta^b$, respectively. Thus
$$ \nabla^be_k = \sum_{j =1}^n \theta^b_{kj} \otimes e_{j}, \ \ \ \Theta^b = d\theta^b - \theta^b \wedge \theta^b, \ \ \ \ \Theta^b_{k\ell } = \sum_{i,j=1}^n R^b_{i\bar{j}k\bar{\ell}} \,\varphi_i \wedge \overline{\varphi}_j, $$
for any $1\leq k, \ell \leq n$ where $\varphi$ is the coframe dual to $e$. Since $\nabla^bT=0$, we get $\nabla^bB=0$, that is,
$$ dB_{i\bar{j}} = \sum_r \big( \theta^b_{ir}B_{r\bar{j}} + B_{i\bar{r}} \overline{\theta^b_{jr}} \big) , $$
or $dB=\theta^bB-B\,\theta^b$ in matrix form. Here we used the fact that  $\theta^b$ is skew-Hermitian. Taking the exterior derivative of both sides of the last equality and using the structure equation, we obtain
\begin{equation} \label{eq:BTheta}
 0 \,= \,\Theta^b B - B \,\Theta^b.
 \end{equation}
The matrix $B$ is Hermitian symmetric, so we may choose a local unitary frame $e$ so that the matrix $B=(B_{i\bar{j}})$ under $e$ is diagonal. Group equal diagonal entries of $B$ into blocks. Equation (\ref{eq:BTheta}) then implies that $\Theta^b_{ij}=0$ whenever $B_{i\bar{i}} \neq B_{j\bar{j}}$, so $\Theta^b$ is block-diagonal. By the symmetry property of $R^b$ we have
$$ S_{i\bar{j}} = \sum_k R^b_{i\bar{j}k\bar{k}} = \sum_k R^b_{k\bar{k}i\bar{j}}, $$
which is the trace of the $(1,1)$-form $\sqrt{-1}\Theta^b_{ij}$ with respect to the K\"ahler form $\omega$ of $g$. Therefore we conclude that $S=(S_{i\bar{j}})= \mbox{tr}_{\omega}(\sqrt{-1}\Theta^b)$, so $S$ is also block diagonal under $e$. Within each block, $B$ is a multiple of the identity matrix, so by a suitable unitary change of basis in each block, we may assume that $S$ is diagonal, while $B$ remains the same. Under the new unitary frame both $B$ and $S$ are diagonalized simultaneously.
\end{proof}

As an immediate consequence of Lemma \ref{lemmaS} and Lemma \ref{lemmadiag}, we have the following lemma:

\begin{lemma} \label{lemmadiag3}
Let $(M^n,g)$ be a balanced BTP manifold with constant Chern holomorphic sectional curvature. Then locally there always exist unitary frames under which $B$, $A$, and $S$ are all diagonal.
\end{lemma}

Finally, by contracting the first two indices $i$ and $j$ in (\ref{eq:stab}) and utilizing the symmetry condition $R^b_{i\bar{j}k\bar{\ell}} = R^b_{k\bar{\ell}i\bar{j}}$ of the Bismut curvature for BTP manifolds, we obtain the following lemma:

\begin{lemma} \label{lemmastabS}
Let $(M^n,g)$ be a BTP manifold, and let $e$ be a local unitary frame under which $S$ is diagonal. Then we have
\begin{equation} \label{eq:stabS}
(S_{i\bar{i}} + S_{k\bar{k}}-S_{j\bar{j}})\,T^j_{ik} = 0, \ \ \ \ \ \ \ \ \ \forall \ 1\leq i,j,k\leq n.
\end{equation}
\end{lemma}

\vspace{0.3cm}

\section{Balanced BTP fourfolds with \texorpdfstring{$r_B<4$}{rB < 4}}

Let $(M^4,g)$ be a balanced BTP manifold such that its Chern holomorphic sectional curvature is equal to a constant $c$. The goal is to show that either $g$ is K\"ahler, in which case $(M^4,g)$ is a complex space form, or $g$ is non-K\"ahler and Chern flat, in which case, if we further assume that the manifold is compact, $(M^4,g)$ is a compact quotient of $\mbox{SL}(2,{\mathbb C}) \times {\mathbb C}$ with the product metric $\lambda g_{\kappa}\times g_0$, where $\lambda g_{\kappa}$ is a constant multiple of the Killing metric and $g_0$ is the flat Euclidean metric. In this section, we will deal with the $r_B<4$ case, and leave the $r_B=4$ case to the next two sections.

We will separate our discussions according to the values of $r_B$, the rank of the $B$ tensor. When $r_B=0$, we have $B=0$ and hence $T=0$, so $g$ is K\"ahler. So we just need to examine the cases $1\leq r_B<4$ in this section. For the most part, the argument works locally.

First let us consider the $r_B=1$ case. In this situation the argument works for all dimensions, and we have the following proposition:

\begin{proposition} \label{proprank1}
Let $(M^n,g)$ be a balanced BTP manifold with constant Chern holomorphic sectional curvature. Then $r_B\neq 1$.
\end{proposition}

\begin{proof}
Assume, for a contradiction, that $r_B=1$. We may assume that $n\geq 2$, as for $n=1$ all Hermitian metrics are K\"ahler (hence $r_B=0$). Choose a local unitary frame $e$ so that $\mbox{ker}(B)$ is spanned by $\{ e_2, \ldots, e_n\}$. So $T^j_{ik}=0$ for all $i,k$ and for all $2\leq j\leq n$. Since $g$ is balanced, $T^1_{1k}=-T^2_{2k}-\cdots-T^n_{nk}=0$ for any $k$. Thus the only possibly non-zero Chern torsion components are $T^1_{ij}$, $2\leq i, j\leq n$. Write $E_{ij}=T^1_{ij}$. Then $E=(E_{ij})$ is a skew-symmetric $(n-1)\times (n-1)$ matrix. Since $T^{\ast}_{1\ast}=0$, by (\ref{eq:Rb}) we have
$$ R^b_{1\bar{1}k\bar{\ell}} = \frac{c}{2}(\delta_{k\ell} + \delta_{1k} \delta_{1\ell }) + \frac{1}{4}\sum_{r=1}^n T^1_{kr} \overline{T^1_{\ell r}} ,$$
hence $R^b_{1\bar{1}1\bar{1}}=c$, $R^b_{1\bar{1}k\bar{1}}=0$ for any $2\leq k\leq n$, and $R^b_{1\bar{1}p\bar{q}}= \frac{c}{2}\delta_{pq} + \frac{1}{4} (E E^{\ast})_{pq}$ for any $2\leq p,q\leq n$, where $E^{\ast}$ stands for the conjugate transpose of $E$. Substituting these expressions into (\ref{eq:stab}) with $i=j=\ell=1$ and $2\leq p,q\leq n$, we obtain
\begin{eqnarray*}
0 & = &  \sum_{r=1}^n \big( R^b_{1\bar{1}r\bar{1}}T^r_{pq} - R^b_{1\bar{1}p\bar{r}} T^1_{rq} -  R^b_{1\bar{1}q\bar{r}} T^1_{pr} \big) \\
& = & R^b_{1\bar{1}1\bar{1}}T^1_{pq} - \sum_{r=2}^n \big( R^b_{1\bar{1}p\bar{r}} T^1_{rq} + R^b_{1\bar{1}q\bar{r}} T^1_{pr} \big)\\
& = & -\frac{1}{4}\big( (EE^{\ast})E + E\,^t\!(EE^{\ast}) \big)_{pq}
\end{eqnarray*}
Since $^t\!E=-E$, the above equality says that $EE^{\ast}E=0$, hence $(EE^{\ast})^2=0$. As $EE^{\ast}$ is a non-negative Hermitian symmetric matrix, this implies that $EE^{\ast}=0$ hence $E=0$. But this means that $T=0$, contradicting the assumption $r_B=1$. Therefore under the balanced BTP and constant Chern holomorphic sectional curvature assumptions, the rank $r_B$ of the $B$ tensor cannot be $1$, and the proof of the proposition is completed.
\end{proof}

Next let us consider the $r_B=2$ case. We will restrict the dimension to $n=4$ now.

\begin{proposition} \label{proprank2}
Let $(M^4,g)$ be a balanced BTP manifold of dimension $4$ with constant Chern holomorphic sectional curvature. Then $r_B\neq 2$.
\end{proposition}

\begin{proof}
Assume, for a contradiction, that $r_B=2$. First let us choose a local unitary frame $e$ so that $\mbox{ker}(B)$ is spanned by $\{ e_3, e_4\}$. Then $T^3_{\ast \ast} = T^4_{\ast \ast}=0$. By the balanced assumption, $T^1_{12}=-T^3_{32}-T^4_{42}=0$. Similarly, $T^2_{12}=0$. Hence the only possibly non-zero torsion components under $e$ are the following:
\begin{equation} \label{eq:rank2}
E = (T^j_{i3})_{1\leq i,j\leq 2} = \left[ \begin{array}{cc} a & r\\p& -a \end{array} \right] , \ F = (T^j_{i4})_{1\leq i,j\leq 2} = \left[ \begin{array}{cc} b & s\\q& -b \end{array} \right],\ T^1_{34}=u, \ T^2_{34}=v.
\end{equation}
The traces of $E$ and $F$ are zero because the metric is balanced. By Lemma \ref{lemmakernelB}, we know that
\begin{equation} \label{eq:rank2a}
R^{b}_{i\bar{j}k\bar{\alpha}} =0,  \ \ \ \ \ \ \ \ \ \ \forall \ k\in \{ 1,2\}, \ \ \forall \ \alpha \in \{ 3,4\}, \ \  \forall \ 1\leq i,j \leq 4.
\end{equation}
On the other hand, since $g$ is BTP with constant Chern holomorphic sectional curvature, if we set \(\ell=\alpha\) in the equation (\ref{eq:Rb}), where \(k\in\{1,2\}\) and \(\alpha\in\{3,4\}\), then we get
\begin{equation} \label{eq:rank2b}
R^{b}_{i\bar{j}k\bar{\alpha}} = \frac{c}{2} \delta_{i\alpha} \delta_{kj} -\frac{1}{2} \sum_{r=1}^2 T^r_{ik}\overline{T^r_{j\alpha}}  + \frac{1}{4} \sum_{r=1}^4 \big( T^j_{kr}\overline{T^i_{\alpha r}} - 3 T^j_{ir}\overline{T^k_{\alpha r}} \big).
\end{equation}
When $j\in \{ 3,4\}$, the above equations (\ref{eq:rank2a}) and (\ref{eq:rank2b}) give us $\sum_r T^r_{ik} \overline{T^r_{34}} =0$. When $i\in \{ 3,4\}$, this leads to
\begin{equation}  \label{eq:rank2-1}
 E \left[ \begin{array}{c} \overline{u} \\ \overline{v} \end{array} \right] \ = \  F \left[ \begin{array}{c} \overline{u} \\ \overline{v} \end{array} \right]  = 0.
\end{equation}
Now let us assume that $j\in \{ 1,2\}$. When $i=k$, (\ref{eq:rank2a}) and (\ref{eq:rank2b}) give us
\begin{equation}  \label{eq:rank2-2}
 E_{1r} \overline{u} = F_{1r} \overline{u} = E_{2r} \overline{v} = F_{2r} \overline{v} = 0, \ \ \ \ \ \ \ \ \ \forall \ 1\leq r\leq 2.
\end{equation}
When $i,j \in \{ 1,2\}$ but $i\neq k$, we get from (\ref{eq:rank2a}) and (\ref{eq:rank2b}) that
$$  T^j_{k\beta} \overline{T^i_{34}} = 3  T^j_{i\beta} \overline{T^k_{34}}, \ \ \ \ \ \ \ \ \ \ \ j\in \{ 1,2\}, \ \beta \in \{ 3,4\}, \ \{i,k\} =\{ 1,2\}.$$
Taking $\beta=3$ and $(i,k)=(1,2)$ or $(2,1)$, we obtain $E_{2j}\overline{u}= 3E_{1j}\overline{v}$ and $E_{1j}\overline{v}= 3E_{2j}\overline{u}$, hence $E_{2j}\overline{u}= E_{1j}\overline{v}=0$. Similarly, by taking $\beta =4$ we get $F_{2j}\overline{u}= F_{1j}\overline{v}=0$. Thus we have
\begin{equation}  \label{eq:rank2-3}
 E_{2r} \overline{u} = E_{1r} \overline{v}=F_{2r} \overline{u} = F_{1r} \overline{v}=0, \ \ \ \ \ \ \ \ \ \forall \ 1\leq r\leq 2.
\end{equation}

We claim that $u=v=0$. To see this, assume on the contrary that $u\neq 0$ or $v\neq 0$. Then by (\ref{eq:rank2-2}) and (\ref{eq:rank2-3}) we have $E=F=0$. In this case the only possibly non-zero torsion components are $T^1_{34}=u$ and $T^2_{34}=v$. A non-trivial linear combination $w$ of $e_1$ and $e_2$ will then satisfy $T^w_{34}=0$ hence $T^w_{\ast \ast }=0$,  leading to $w\in \mbox{ker}(B)$, which is a contradiction to the assumption that $r_B=2$. So we know that we must have $u=v=0$.

Now let us come back to (\ref{eq:rank2a}) and (\ref{eq:rank2b}). Taking $j\in\{1,2\}$ and $i=\alpha$, we obtain
\begin{equation}  \label{eq:rank2-4}
 2cI = 3 E^{\ast}E - 2 E E^{\ast}, \ \ \ \ \ 2cI = 3 F^{\ast}F - 2 F F^{\ast} .
\end{equation}
Similarly, if we set $j\in \{ 1,2\}$ and $i\in \{ 3,4\}$ but $i\neq \alpha$, then we have
\begin{equation}  \label{eq:rank2-5}
  3 E^{\ast}F = 2 F E^{\ast}, \ \ \ \ \ 3 F^{\ast}E = 2 E F^{\ast} .
\end{equation}
Finally, by taking traces in (\ref{eq:rank2-4}), we get
\begin{equation}  \label{eq:rank2tr}
 4c \ = \, \ \parallel\! E\! \parallel^2 \ \, = \, \ \parallel\! F\! \parallel^2 .
\end{equation}

Let us proceed with the proof of Proposition \ref{proprank2}. We claim that $c$ must be $0$. Assume on the contrary that $c>0$. By writing out the elements of $E$ and $F$ in the equations (\ref{eq:rank2-4}) and (\ref{eq:rank2-5}), we obtain
\begin{eqnarray}
 && |p|^2=|r|^2 = 2c - |a|^2, \ \ \ \ \ \ |q|^2=|s|^2 = 2c-|b|^2, \label{eq:a}\\
 && a \bar{p} = \bar{a}r, \ \ \ \ \ \ b\bar{q} = \bar{b}s, \label{eq:b} \\
 && p\bar{q} = r\bar{s} = -a\bar{b}, \ \ \ \ \ \ a\bar{q}=\bar{b}r, \ \ \ \ \ \ a\bar{s}=\bar{b}p.   \label{eq:c}
\end{eqnarray}

Consider the row vectors $(a,p)$ and $(b,q)$ in ${\mathbb C}^2$. Equation (\ref{eq:a}) says that they have length $\sqrt{2c}>0$. These two vectors are perpendicular to each other because $a\bar{b}+p\bar{q}=0$ from the first equality of (\ref{eq:c}). Therefore these two vectors cannot be proportional to each other, namely, $aq-bp\neq 0$.

If $a=0$, then by the first equality of (\ref{eq:a}) we have $pr\neq 0$. From the first equality in (\ref{eq:c}), we then get $q=0$.  So by the second equality in (\ref{eq:a}) we know that $b\neq 0$. Now the  second equality of (\ref{eq:c}), $a\bar{q}=\bar{b}r$, would give us a contradiction.

If $a\neq 0$, then the first equality of (\ref{eq:b}) gives us $r=(a\bar{p})/\bar{a}$. Substituting this into the second equality of (\ref{eq:c}) gives $aq-bp=0$, a contradiction.

The above discussion indicates that the assumption $c>0$ always leads to a contradiction, so we must have $c=0$. By (\ref{eq:rank2tr}), we conclude that  $E=F=0$, hence $T=0$, violating the $r_B=2$ assumption. So for $n=4$ the $r_B=2$ case cannot occur. This completes the proof of the proposition.
\end{proof}

\begin{proposition} \label{proprank3}
Let $(M^4,g)$ be a compact balanced BTP manifold of dimension $4$ with constant Chern holomorphic sectional curvature and with $r_B= 3$. Then $g$ is Chern flat. Furthermore, the universal cover of $(M^4,g)$ is holomorphically isometric to the product $(\mbox{SL}(2,{\mathbb C}), \lambda g_{\kappa}) \times ({\mathbb C},g_0)$, where $g_0$ is the flat metric on ${\mathbb C}$, $\lambda >0$ is a constant, and $g_{\kappa}$ is the standard Killing metric.
\end{proposition}

\begin{proof}
Let $e$ be a local unitary frame so that $e_4$ spans the kernel of $B$. We have $T^4_{\ast \ast }=0$. Let $E=(E_{ij})=(T^j_{i4})_{1\leq i,j\leq 3}$ be the corresponding $3\times 3$ matrix.  By (\ref{eq:Rb}) we have
\begin{equation*}
 R^b_{4 \bar{j} k\bar{4}}= \frac{c}{2}\delta_{kj} - \frac{1}{2}\sum_r T^r_{4k} \overline{ T^r_{j4}} - \frac{3}{4} \sum_r T^j_{4r} \overline{ T^k_{4r}}, \ \ \ \ \ \ \ \forall \ 1\leq j,k\leq 3,
\end{equation*}
On the other hand, for any $1\leq k\leq 3$, since $e_4\in \mbox{ker}(B)$ and $e_k \perp \mbox{ker}(B)$, by Lemma \ref{lemmakernelB} we get $R^b_{\ast \bar{\ast} k\bar{4}}=0$, hence $R^b_{4 \bar{j} k\bar{4}}=0$ for any $1\leq j,k\leq 3$ which in matrix form is
\begin{equation} \label{eq:EE}
 2cI_3 \,= \,3 E^{\ast}E - 2E E^{\ast}.
\end{equation}
Taking the trace in the above line, we get $6c = \, \parallel\!E\!\parallel^2\ \geq 0$. When $c=0$, we get $E=0$. In the following let us assume that $c>0$.

Again by (\ref{eq:Rb}) we have $R^b_{4 \bar{4} 4\bar{4}}=c$, $R^b_{4 \bar{4} 4\bar{j}}=0$ for any $1\leq j\leq 3$, and
$$  R^b_{4 \bar{4} i\bar{j}}= \frac{c}{2}\delta_{ij} - \frac{1}{2}\sum_r T^r_{i4} \overline{T^r_{j4} } + \frac{1}{4}\sum_r T^j_{4r} \overline{T^i_{4r} }, \ \ \ \ \ \ \forall \ 1\leq i,j\leq 3. $$
Let $D=(D_{ij})=(R^b_{4\bar{4}i\bar{j}})_{1\leq i,j\leq 3}$. Then the above equation becomes
\begin{equation} \label{eq:D}
 D \,= \,\frac{1}{4}\big( 2cI - 2E E^{\ast}+  E^{\ast}E\big) .
\end{equation}
Applying (\ref{eq:stab}) with $i=j=q=4$ and $p,\ell<4$ gives
$$ 0 = \sum_r R^{b}_{4\bar{4}r\bar{\ell}}T^r_{p4} - \sum_r R^{b}_{4\bar{4}p\bar{r}}T^{\ell}_{r4} - \sum_s R^{b}_{4\bar{4}4\bar{s}}T^{\ell}_{ps}. $$
Since $R^b_{4\bar{4}4\bar{j}}=0$ for any $j<4$, in the above equality the index $r<4$ and $s=4$, so we have
\begin{equation} \label{eq:DE}
 ED - DE - cE  =0.
\end{equation}
By (\ref{eq:EE}),  $EE^{\ast}$ and $E^{\ast}E$ can be expressed in terms of each other. Plugging this into (\ref{eq:D}) gives
$$ D = cI - \frac{1}{2}E^{\ast }E = \frac{2c}{3}I - \frac{1}{3} EE^{\ast}. $$
Therefore by (\ref{eq:DE}) we get
$$ 0= E\big(cI - \frac{1}{2}E^{\ast }E\big) - \big( \frac{2c}{3}I - \frac{1}{3} EE^{\ast}\big) E - cE = -\frac{2c}{3}E - \frac{1}{6}EE^{\ast}E. $$
Multiplying the last equality from the left by $E^\ast$ gives
$$
E^\ast E(E^\ast E+4cI)=0.
$$
Since $E^\ast E$ is non-negative Hermitian and $c>0$, this implies that $E^\ast E=0$, hence $E=0$. Now by (\ref{eq:EE}) we get $c=0$, contradicting the assumption that $c>0$. So $c$ must be $0$ and $E=0$.

We have proved that $T^{\ast}_{\ast 4}=0$, and $T^4_{\ast \ast}=0$ by our choice of $e_4$. We claim that $L:=\mbox{ker}(B)= {\mathbb C}e_4$ will be preserved under the Levi-Civita connection $\nabla^g$. To see this, recall that under any local unitary frame $e$ with dual coframe $\varphi$ we always have (see for instance \cite{YangZ} or \cite{ZhaoZ24})
$$ \nabla^ge_i = \nabla^be_i - \frac{1}{2}\sum_{j,k} \big( T^j_{ik}\varphi_k - \overline{T^i_{jk} } \overline{\varphi}_k \big) e_j + \frac{1}{2} \sum_{j,k}T^k_{ij} \overline{\varphi}_k \overline{e}_j . $$
Hence in our case $\nabla^ge_4 = \nabla^be_4 = \theta^b_{44}e_4 \in L$. Therefore $L$ gives a flat de Rham factor and the universal cover of $M^4$ is holomorphically isometric to the product $N^3\times {\mathbb C}$.

We now show that the induced metric on $N^3$ is also balanced and BTP.  Since all torsion components involving the index \(4\) vanish, the torsion of \(M\) restricts exactly to the torsion of \(N\).  The Bismut connection also splits, so the induced metric on \(N\) is BTP.  The balanced condition restricts as well, because for \(1\le a\le3\),
$$
\sum_{b=1}^3T^b_{ba}=\sum_{b=1}^4T^b_{ba}=0.
$$
Moreover, the \(B\)-tensor of \(N\) is the restriction of \(B\) to the \(N\)-factor, hence \(r_B(N)=3\). The induced metric on \(N\) is balanced and BTP, and its Chern holomorphic sectional curvature is still \(c=0\). Since \(r_B(N)=3\), the local rank-three case in the proof of \cite[Proposition 2.4, Case 2]{ZhaoZ25} applies to \(N\); this part of the argument does not use compactness and implies that \(N\) is Chern flat. The product Chern connection on \(N\times {\mathbb C}\) is the direct sum of the Chern connections of the two factors, so \(\widetilde M\), and hence \(M\), is Chern flat. Applying the compact Chern-flat BTP classification \cite[Theorem 1.2 and Theorem 3.4]{PodestaZ} to the original compact fourfold \(M\), and using \(r_B=3\), we obtain
$$
\widetilde M\cong(\mbox{SL}(2,{\mathbb C}),\lambda g_{\kappa})\times({\mathbb C},g_0)
$$
for some \(\lambda>0\). This completes the proof of the proposition.
\end{proof}

Propositions \ref{proprank1}, \ref{proprank2}, and \ref{proprank3} prove Theorem \ref{thm1} in the $r_B<4$ case. In the next two sections, we will rule out the possibility of $r_B=4$.

\vspace{0.3cm}

\section{Balanced BTP fourfolds with \texorpdfstring{$r_B=4$ and $S\neq 0$}{rB = 4 and S != 0}}

The goal of this section is to rule out the possibility of $r_B=4$ and $S\neq 0$ in Theorem \ref{thm1}, as stated in the following proposition:

\begin{proposition} \label{prop4a}
Let $(M^4,g)$ be a balanced BTP manifold with constant Chern holomorphic sectional curvature $c$. Then the following case cannot occur: $r_B=4$ and $S\neq 0$.
\end{proposition}

Here $S\neq 0$ means that the first Bismut Ricci curvature $S$ is not identically zero. In the following we will assume  that $(M^4,g)$ is a balanced BTP fourfold with constant Chern holomorphic sectional curvature.  Assume, for a contradiction, that $r_B=4$ and $S\neq 0$. By (\ref{eq:S}),
$$ S = \frac{5c}{2}I + \frac{1}{4}(B-A), $$
where $A_{i\bar{j}} = \sum_{r,s} T^r_{is} \overline{T^r_{js}} $ and $B_{i\bar{j}} = \sum_{r,s} T^j_{rs} \overline{T^i_{rs}} $. Hence $\mbox{tr}(S)=10c$. Since $\nabla^bS=0$, the eigenvalues of $S$ are all constants. Also, by Lemma \ref{lemmadiag3}, there exists a local unitary frame $e$ so that $A$, $B$, and $S$ are simultaneously diagonal. Under such a frame, one has (\ref{eq:stabS}):
\begin{equation}  \label{eq:stabS2}
 \big(S_{i\bar i}+S_{k\bar k}-S_{j\bar j}\big)T^j_{ik}=0,
\qquad 1\le i,j,k\le4.
\end{equation}

\vspace{0.25cm}

\noindent {\bf Claim 1.} $S$ must have at least one zero eigenvalue.
\begin{proof}
Denote by $\lambda_1\geq \lambda_2\geq \lambda_3\geq  \lambda_4$ the four eigenvalues of $S$. We want to show that at least one of them must be zero. Assume the contrary. Then one of the following cases occurs; we rule them out in turn.

\vspace{0.1cm}
(1) $\lambda_4>0$ or $\lambda_3> 0 > \lambda_4$.

In this case, for any $1\leq p<q\leq 4$, we have $\lambda_p+(\lambda_q-\lambda_4)\geq \lambda_p>0$, so by (\ref{eq:stabS2}) we have $T^4_{pq}=0$ for any $p$, $q$, thus $e_4$ lies in the kernel of $B$, contradicting the assumption that $r_B=4$.

\vspace{0.1cm}
(2) $0>\lambda_1$ or $\lambda_1 > 0>\lambda_2$.

In this case we have for any $1\leq p<q\leq 4$ that
$\lambda_p+\lambda_q-\lambda_1 = (\lambda_p -\lambda_1) + \lambda_q \leq \lambda_q <0$, hence $T^1_{pq}=0$ for any $p$, $q$, violating $r_B=4$.

\vspace{0.1cm}
(3) $\lambda_2> 0 > \lambda_3$.

Let $1\leq p<q\leq 4$. If $q>2$, we have $\lambda_p+\lambda_q-\lambda_1\leq \lambda_q <0$. If $q=2$, then $p=1$, and we have $\lambda_p+\lambda_q-\lambda_1= \lambda_1+\lambda_2-\lambda_1 = \lambda_2 >0$. So by (\ref{eq:stabS2}) we always have $T^1_{pq}=0$, a contradiction. This completes the proof of the claim.
\end{proof}

\vspace{0.25cm}

\noindent {\bf Claim 2.} $S$ cannot have exactly one zero eigenvalue.
\begin{proof}
Assume on the contrary that the Bismut Ricci tensor $S$ has exactly one zero eigenvalue. Let $S=\mbox{diag}\{ x,y,z,0\}$ with $xyz\neq 0$.  Since \(r_B=4\), the matrix \(B\) is positive definite.  In particular, some component \(T^4_{pq}\) is nonzero.  Applying \eqref{eq:stabS2} to this component gives
$$
S_{p\bar p}+S_{q\bar q}=0.
$$
Here neither \(p\) nor \(q\) can be \(4\), since otherwise one of \(x,y,z\) would be zero.  Thus two of the three nonzero eigenvalues of \(S\) are opposite.  After relabeling the first three basis vectors, we may write
$$
S=\operatorname{diag}(a,b,-a,0),\qquad a>0,\quad b\neq0.
$$
Put \(\lambda=(\lambda_1,\lambda_2,\lambda_3,\lambda_4)=(a,b,-a,0)\).  By \eqref{eq:stabS2}, a component \(T^j_{ik}\) can be nonzero only if
$$
\lambda_i+\lambda_k-\lambda_j=0.
$$

We now consider the two possible signs of \(b\).  If \(b>0\), then the only lower pair whose eigenvalue sum is \(-a\) is \((3,4)\).  Hence the only possible component with upper index \(3\) is \(T^3_{34}\).  Since \(B_{3\bar3}>0\), we must have
$$
T^3_{34}\neq0.
$$
Applying the stabilizer identity \eqref{eq:stab} to \((R^b\cdot T)^3_{34}\), we get
$$
0 \ =\ R^b_{i\bar j3\bar3}T^3_{34}
-R^b_{i\bar j3\bar3}T^3_{34}
-R^b_{i\bar j4\bar4}T^3_{34} \ = \ -R^b_{i\bar j4\bar4}T^3_{34}.
$$
Therefore \(R^b_{i\bar j4\bar4}=0\) for all \(i,j\).

If \(b<0\), then the only lower pair whose eigenvalue sum is \(a\) is \((1,4)\).  Hence the only possible component with upper index \(1\) is \(T^1_{14}\).  Since \(B_{1\bar1}>0\), we have
$$
T^1_{14}\neq0.
$$
Applying \eqref{eq:stab} to \((R^b\cdot T)^1_{14}\), we get
$$
0 \ = \ R^b_{i\bar j1\bar1}T^1_{14}
-R^b_{i\bar j1\bar1}T^1_{14}
-R^b_{i\bar j4\bar4}T^1_{14}\ = \ -R^b_{i\bar j4\bar4}T^1_{14}.
$$
Thus again \(R^b_{i\bar j4\bar4}=0\) for all \(i,j\).

In either case, \(R^b_{4\bar44\bar4}=0\).  On the other hand, \eqref{eq:stabS2} gives \(T^4_{4r}=0\) for every \(r\): for \(r=1,2,3\) the coefficient is \(S_{r\bar r}\neq0\), and \(T^4_{44}=0\) by skew-symmetry.  Therefore, by \eqref{eq:Rb},
$$
R^b_{4\bar44\bar4} \ = \ c-\sum_r|T^4_{4r}|^2 \ = \ c.
$$
So \(c=0\).  But then the trace identity gives \(\operatorname{tr}S=10c=0\), whereas
$$
\operatorname{tr}S=a+b-a+0=b\neq0.
$$
This contradiction proves the claim.
\end{proof}

\vspace{0.25cm}

\noindent {\bf Claim 3.} $S$ cannot have exactly three zero eigenvalues.
\begin{proof}
Assume on the contrary that  $S$ has exactly three zero eigenvalues. Without loss of generality, we may suppose that $S=\mbox{diag}\{ 0,0,0,x\}$ with $x\neq 0$ under a suitable local unitary frame $e$. We have $10c=\mbox{tr}(S)=x\neq  0$. By (\ref{eq:stabS2}), we know that the only possibly non-zero torsion components are $T^j_{ik}$ and $T^4_{i4}$ for $1\leq i,j,k\leq 3$. By taking  a unitary change in $\{ e_1,e_2, e_3\}$, we may assume that $T^4_{24}=T^4_{34}=0$. So $T^4_{14}\neq 0$ as otherwise $e_4$ will belong to the kernel of $B$. The component $T^4_{14}$ is the only non-trivial torsion component involving the index $4$. Below we will repeatedly use the identity (\ref{eq:stab}), which will be denoted as $(R\cdot T)^{\ell}_{pq}$. First consider $(R\cdot T)^{4}_{14}$, which gives us $R^b_{i\bar{j}1\bar{1}}T^4_{14}=0$. Hence $R^b_{i\bar{j}1\bar{1}}=0$ for any $i$, $j$. Write $a=T^4_{14}$. By (\ref{eq:Rb}), we have
$$ R^b_{1\bar{1}1\bar{1}}=c-\sum_r |T^1_{1r}|^2, \ \ \ R^b_{4\bar{4}1\bar{1}}= \frac{c}{2} - \frac{1}{4} |a|^2.$$
So $2c=|a|^2>0$ and $|T^1_{12}|^2 + |T^1_{13}|^2 =c >0$. Next,  let us compute by (\ref{eq:Rb}) that
\begin{eqnarray*}
R_{4\bar{1}}& := &  R^b_{1\bar{4}4\bar{1}} \ = \  0\,,\\
 R_{4\bar{2}}& := &  R^b_{1\bar{4}4\bar{2}} \ = \  \frac{a}{4} \overline{T^1_{12} }  \,,\\
 R_{4\bar{3}}& := &  R^b_{1\bar{4}4\bar{3}} \ = \ \frac{a}{4} \overline{T^1_{13} }\,.
\end{eqnarray*}
Meanwhile, applying (\ref{eq:stab}) to $(R\cdot T)^j_{24}$ and $(R\cdot T)^j_{34}$ for $1\leq j\leq 3$, we obtain
$$ T^j_{12}R_{4\bar{1}} =  T^j_{23}R_{4\bar{3}},\ \ \ \ \ \  \ \ \ \ T^j_{13}R_{4\bar{1}} =  -T^j_{23}R_{4\bar{2}}.$$
Substituting the above values for $R_{4\bar{1}}$, $R_{4\bar{2}}$, and $R_{4\bar{3}}$, we obtain $T^j_{23} \overline{T^1_{12}} = T^j_{23} \overline{T^1_{13}} = 0$. Since $T^1_{12}$ and $T^1_{13}$ cannot both be zero, $T^j_{23}=0$ for each $1\leq j\leq 3$. On the other hand, by the balanced assumption on the metric, we have
$$ T^1_{12}=-T^3_{32}=0, \ \ \ \ \ \ T^1_{13} = -T^2_{23}=0, $$
which is a contradiction. This shows that the case when $S$ has exactly three zero eigenvalues cannot occur, and the claim is proved.
\end{proof}

\vspace{0.25cm}

\noindent {\bf Claim 4.} $S$ cannot have exactly two zero eigenvalues.
\begin{proof}
Assume on the contrary that  $S$ has exactly two zero eigenvalues. Without loss of generality, we may suppose that $S=\mbox{diag}\{ 0,0,x,y\}$ with $xy\neq 0$ under a suitable local unitary frame $e$. We will divide the discussion into two cases, depending on whether the trace $10c=x+y$ is zero or not.

\vspace{0.15cm}

\noindent {\em Case 1:} $x+y\neq 0$.

In this case by (\ref{eq:stabS}) we know that the only possibly non-zero torsion components are given below:
$$ T^1_{12}, \ T^2_{12}, \ T^3_{i3}, \ T^4_{i4},\ \ \mbox{and} \ T^3_{i4}, \ T^4_{i3} \ \ \mbox{when} \ \ x=y. $$
Here $1\leq i\leq 2$, and $T^1_{12}$ and $T^2_{12}$ are non-zero since $r_B=4$. By definition, we have
$$ B_{1\bar{1}}=\sum_{r,s=1}^4|T^1_{rs}|^2 = 2 |T^1_{12}|^2, \ \ \ \ B_{2\bar{2}} = 2 |T^2_{12}|^2, \ \ \ \ B_{1\bar{2}} = 2 T^2_{12} \overline{T^1_{12} } . $$
Therefore we have $B_{1\bar{1}} B_{2\bar{2}} = |B_{1\bar{2}}|^2$. But this contradicts the assumption $r_B=4$, which means that $B$ is positive definite. Hence this case does not occur.

\vspace{0.15cm}

\noindent {\em Case 2:} $x+y= 0$.

In this case again by (\ref{eq:stabS}) we know that the only possibly non-zero torsion components are given below:
$$ T^1_{12}, \ T^2_{12}, \ T^3_{13}, \ T^3_{23}, \ T^4_{14},\ T^4_{24}, \  T^1_{34}, \ T^2_{34}. $$
Keeping $e_3$ and $e_4$ fixed, we may make a unitary change of $\{e_1,e_2\}$ so that $T^2_{34}=0$.  Note that now $T^1_{34}\neq 0$, as otherwise a linear combination of $e_1$ and $e_2$ would lie in the kernel of $B$, violating the $r_B=4$ assumption. Also, $T^2_{12}\neq 0$ as otherwise $e_2$ would be in the kernel of $B$. Applying (\ref{eq:stab}) to $(R\cdot T)^2_{34}$ gives $R^b_{i\bar{j}1\bar{2}}T^1_{34}=0$, hence $R^b_{i\bar{j}1\bar{2}}=0$ for any $i$, $j$. Applying the same identity to $(R\cdot T)^2_{12}$ and using $R^b_{i\bar{j}1\bar{2}}=0$ gives $R^b_{i\bar{j}1\bar{1}}T^2_{12}=0$, hence $R^b_{i\bar{j}1\bar{1}}=0$ for any $i$, $j$. On the other hand, since $c=0$, by (\ref{eq:Rb}) we compute
\begin{eqnarray*}
 R^b_{2\bar{2}1\bar{1}} & = & \sum_{r=1}^4 \big( -\frac{1}{2}  |T^r_{12}|^2 +\frac{1}{4} |T^1_{2r}|^2 + \frac{1}{4} |T^2_{1r}|^2 -\frac{3}{4} T^1_{1r} \overline{ T^2_{2r}} -\frac{3}{4} T^2_{2r} \overline{ T^1_{1r}}  \big) \\
 & = & - \frac{1}{4}|T^1_{12}|^2 - \frac{1}{4}|T^2_{12}|^2.
 \end{eqnarray*}
Therefore we obtain $T^1_{12}=T^2_{12}=0$, contradicting our assumption that $T^2_{12}\neq 0$. This concludes the proof of the case and the claim.
\end{proof}

Together, Claims 1--4 prove Proposition \ref{prop4a}. It remains to deal with the case $r_B=4$ while $S=0$ identically. This will be treated in the next section.

\vspace{0.3cm}

\section{Balanced BTP fourfolds with \texorpdfstring{$r_B=4$ and $S=0$}{rB = 4 and S = 0}}

We now rule out the case $r_B=4$ and $S=0$, completing the proof of Theorem \ref{thm1}. More precisely, we have the following proposition:

\begin{proposition} \label{prop4b}
Let $(M^4,g)$ be a balanced BTP manifold with constant Chern holomorphic sectional curvature. Then the following case cannot occur: $r_B=4$ and $S= 0$.
\end{proposition}

In the following we will assume that $(M^4,g)$ is a balanced BTP manifold with constant Chern holomorphic sectional curvature $c$. Assume, for a contradiction, that $r_B=4$ and $S=0$. By (\ref{eq:S}) and (\ref{eq:traceS}) we have $c=0$ and $A=B$. Let $e$ be a local unitary frame so that $B$ is diagonal under $e$. We first prove the following claim:

\vspace{0.4cm}

\noindent {\bf Claim 1:} $B$ cannot have any eigenvalue with multiplicity $1$.
\begin{proof} Assume on the contrary that $B$ does have an eigenvalue with multiplicity $1$. Without loss of generality, we may assume that $e_1$ corresponds to that eigenvalue. The assumption means that $B$ is diagonal and $B_{1\bar{1}}\neq B_{i\bar{i}}$ for $2\leq i\leq 4$. By (\ref{eq:BTheta}), we know that $\big( B_{i\bar{i}} - B_{j\bar{j}}\big) \Theta^b_{ij} =0$ holds for any $i$, $j$. So in our case we have $\Theta^b_{1i}=0$ for any $2\leq i\leq 4$. That is, $R^b_{\ast \bar{\ast}1\bar{i}}=0$. Since $c=0$, by (\ref{eq:Rb}) we have
$$ 0 = R^b_{i\bar{1}1\bar{i}} = \sum_{r=1}^4 \big\{ \frac{1}{2} |T^r_{1i}|^2 + \frac{1}{4} (T^1_{1r} \overline{T^i_{ir}} + T^i_{ir} \overline{T^1_{1r}})  -\frac{3}{4} ( |T^i_{1r}|^2 + |T^1_{ir}|^2 ) \big\}.$$
Therefore
 $$ \sum_r (T^1_{1r} \overline{T^i_{ir}} + T^i_{ir} \overline{T^1_{1r}}) = 3 \sum_r ( |T^i_{1r}|^2 + |T^1_{ir}|^2 ) - 2\sum_r |T^r_{1i}|^2, \ \ \ \ \ \ \forall \ 2\leq i\leq 4.
 $$
Substituting this into the expression for $R^b_{1\bar{1}i\bar{i}}$ in (\ref{eq:Rb}), we obtain
\begin{eqnarray}
R^b_{1\bar{1}i\bar{i}} & = & \sum_{r=1}^4 \big\{ -\frac{1}{2} |T^r_{1i}|^2 - \frac{3}{4} (T^1_{1r} \overline{T^i_{ir}} + T^i_{ir} \overline{T^1_{1r}})  +\frac{1}{4} ( |T^i_{1r}|^2 + |T^1_{ir}|^2 ) \big\} \nonumber \\
& = & \sum_{r=1}^4 \big\{ -\frac{1}{2} |T^r_{1i}|^2 - \frac{3}{4} \big( 3( |T^i_{1r}|^2 + |T^1_{ir}|^2 ) - 2 |T^r_{1i}|^2 \big)   +\frac{1}{4} ( |T^i_{1r}|^2 + |T^1_{ir}|^2 ) \big\} \nonumber \\
& = & \sum_{r=1}^4 \big\{  |T^r_{1i}|^2 - 2 ( |T^i_{1r}|^2 + |T^1_{ir}|^2 ) \big\}, \ \ \ \ \ \ \ \ \ \ \ \ \  \ \ \ \ \ \ \ \ \ \ \ \ \ \forall \ 2\leq i\leq 4. \label{eq:bb}
\end{eqnarray}
Let us write $a:=\sum_r|T^1_{1r}|^2$. Then by (\ref{eq:Rb}) again we have $R^b_{1\bar{1}1\bar{1}}=-a$, while by definition, $B_{1\bar{1}}=\sum_{r,s} |T^1_{rs}|^2$ and $A_{1\bar{1}}=\sum_{r,s} |T^r_{1s}|^2$, so if we sum over $i$ from $2$ to $4$ in (\ref{eq:bb}) and use the fact that $S=0$, we obtain
$$ a=-R^b_{1\bar{1}1\bar{1}} = \sum_{i=2}^4 R^b_{1\bar{1}i\bar{i}}  = A_{1\bar{1}} -2(A_{1\bar{1}}-a) -2(B_{1\bar{1}}-a) =4a-3B_{1\bar{1}}.$$
In the last equality, we used the fact that $A=B$. So we have $B_{1\bar{1}}=a$. On the other hand, by definition,
$$ B_{1\bar{1}} =\sum_{r,s=1}^4 |T^1_{rs}|^2 = 2\sum_{r=1}^4|T^1_{1r}|^2 + \sum_{r,s=2}^4 |T^1_{rs}|^2 = 2a + \sum_{r,s=2}^4 |T^1_{rs}|^2. $$
Hence $B_{1\bar{1}}=a$ implies that $T^1_{rs}=0$ for any $r$, $s$, which leads to $B_{1\bar{1}}=0$, a contradiction. This completes the proof of the claim.
\end{proof}

Since $B$ has no eigenvalue of multiplicity $1$ and the dimension is $4$, $B$ has either two distinct eigenvalues, each of multiplicity $2$, or a single eigenvalue of multiplicity $4$. In the latter case, $B=\lambda I$ is a scalar multiple of the identity. We consider these two cases separately.

\vspace{0.4cm}

\noindent {\bf Claim 2:} $B$ cannot have two distinct eigenvalues, each with multiplicity $2$.

\begin{proof}
Assume on the contrary that $B$ does have two distinct eigenvalues, each with multiplicity $2$. Choose a unitary frame $e$ so that $B=\mbox{diag}\{ x,x,y,y\}$, where $x>y>0$. Set $P=\{1,2\}$ and $Q=\{3,4\}$. For $1\leq a\leq 4$, write $u_a=\sum_{i=1}^2 T^i_{ia}$. Then by the assumption that the metric is balanced, we have
$$ u_1 = - T^2_{12}, \ \ u_2 = T^1_{12}, \ \ u_3=T^4_{34}, \ \ u_4=-T^3_{34}. $$
In particular, $u_r=-\sum_{\alpha =3}^4 T^{\alpha}_{\alpha r}$. Since $x\neq y$, by (\ref{eq:BTheta}) we have $\Theta^b_{i\alpha}=0$ for any $i\in P$ and any $\alpha \in Q$. Let us denote by $i'$ the element in $P$ other than $i$, and by $\alpha'$ the element in $Q$ other than $\alpha$. Then by (\ref{eq:Rb}) we have
\begin{eqnarray*}
 0 & = &  R^b_{i\bar{\alpha} \alpha \bar{i}} \ = \ \sum_{r=1}^4 \big\{ \frac{1}{2} |T^r_{i\alpha}|^2 + \frac{1}{4} \big( T^i_{ir} \overline{T^{\alpha}_{\alpha r} } + T^{\alpha}_{\alpha r} \overline{T^i_{ir}} \big) -\frac{3}{4} \big( |T^{\alpha}_{ir}|^2 + |T^i_{\alpha r}|^2 \big) \big\} \\
 & = & \frac{1}{2}\big\{ |T^i_{i\alpha}|^2 + |T^{i'}_{i\alpha}|^2 + |T^{\alpha}_{i\alpha}|^2+|T^{\alpha'}_{i\alpha}|^2 \big\} + \frac{1}{4} \sum_{r=1}^4 \big\{ T^i_{ir} \overline{T^{\alpha}_{\alpha r} } + T^{\alpha}_{\alpha r} \overline{T^i_{ir}} \big\} \,+ \\
 & & - \frac{3}{4} \big\{ |T^{\alpha}_{ii'}|^2 + |T^{\alpha}_{i\alpha}|^2 + |T^{\alpha}_{i\alpha'}|^2 + |T^i_{i\alpha }|^2  + |T^i_{i'\alpha }|^2  + |T^i_{\alpha \alpha'}|^2 \big\} \\
 & = & \frac{1}{2}\big\{  |T^{i'}_{i\alpha}|^2 + |T^{\alpha'}_{i\alpha}|^2 \big\} -\frac{1}{4}  \big\{ |T^i_{i\alpha}|^2 +  |T^{\alpha}_{i\alpha}|^2  \big\} + \frac{1}{4} \sum_{r=1}^4 \big\{ T^i_{ir} \overline{T^{\alpha}_{\alpha r} } + T^{\alpha}_{\alpha r} \overline{T^i_{ir}} \big\} \,+ \\
 & & - \frac{3}{4} \big\{  |T^{\alpha}_{i\alpha'}|^2 + |T^i_{i'\alpha }|^2  \big\} - \frac{3}{4} \big\{ |T^{\alpha}_{12}|^2     + |T^i_{34}|^2 \big\} \\
 & = & \big\{ \frac{1}{2}  |T^{i'}_{i\alpha}|^2 -  \frac{3}{4} |T^i_{i'\alpha }|^2  \big\} + \big\{ \frac{1}{2}  |T^{\alpha'}_{i\alpha}|^2  -  \frac{3}{4} |T^{\alpha}_{i\alpha'}|^2 \big\} -\frac{1}{4}  \big\{ |T^i_{i\alpha}|^2 +  |T^{\alpha}_{i\alpha}|^2  \big\}\\
 & & - \frac{3}{4} \big\{ |T^{\alpha}_{12}|^2     + |T^i_{34}|^2 \big\} + \frac{1}{4} \sum_{r=1}^4 \big\{ T^i_{ir} \overline{T^{\alpha}_{\alpha r} } + T^{\alpha}_{\alpha r} \overline{T^i_{ir}} \big\}.
 \end{eqnarray*}
If we sum over $i\in P$ and $\alpha\in Q$, then the last term becomes $-\frac{1}{2}\sum_r |u_r|^2$, while the other brackets all remain non-positive, so each bracketed sum must be zero. This actually gives us $T=0$, which contradicts the assumption that $r_B=4$, so the claim is proved.
\end{proof}

For the proof of Proposition \ref{prop4b}, after the above two claims the only case left is when $B$ has one eigenvalue of multiplicity $4$, or equivalently, when $B$ is a scalar multiple of the identity. So Proposition \ref{prop4b} will be proved after we prove the following proposition:

\begin{proposition} \label{prop4c}
Let $(M^4,g)$ be a balanced BTP manifold with constant Chern holomorphic sectional curvature. Then the following case cannot occur: $S= 0$ and $B=\lambda I$ for some constant $\lambda >0$.
\end{proposition}

\begin{proof}
Let $(M^4,g)$ be a balanced BTP manifold with constant Chern holomorphic sectional curvature. Assume that  $S=0$ and $B=\lambda I$ for a constant $\lambda >0$. We will derive a contradiction. After rescaling the metric if necessary, we may assume that $B=2I$. We first claim the following:

\vspace{0.2cm}

\noindent {\bf Claim A:} There exists a non-empty open subset $U\subset M^4$ and a unitary frame $e$ in $U$ such that
\begin{equation} \label{eq:12-34=0}
T^{\ast}_{12}=T^{\ast}_{34}=0.
\end{equation}

Start with any local unitary frame $e$ on an open subset $U_0\subset M^4$. We will make a unitary change of $e$ on a smaller open subset $U\subset U_0$ so that (\ref{eq:12-34=0}) holds. Denote by $\varphi$ the coframe dual to $e$, and consider the rank $6$ bundle $W=\Lambda^2(T\!M^{\ast})$ where $T\!M^{\ast}$ is the holomorphic cotangent bundle. The metric $g$ induces a metric $\langle , \rangle$ on $W$. For $1\leq j\leq 4$, denote by $\alpha_j=\sum_{1\leq i<k\leq 4} T^j_{ik}\,\varphi_i\wedge \varphi_k$ the local $(2,0)$-forms  which are sections of $W$. Note that at any point $p\in M$, $W_p$ can be identified with the space of skew-symmetric $4\times 4$ matrices, with the metric given by
$$ \langle X,\overline{Y} \rangle = \frac{1}{2}\mbox{tr}(XY^{\ast}),$$
where $X$ and $Y$ are skew-symmetric $4\times 4$ matrices and $Y^{\ast}$ denotes the conjugate transpose of $Y$. Since $B=2I$, each $\alpha_j$ is of unit length and they are mutually perpendicular to each other. So they span a rank $4$ subbundle $W_1\subset W$. Denote by $W_1^{\perp}$ the rank $2$ subbundle of $W$ which is the orthogonal complement of $W_1$. Pick two local $(2,0)$-forms
$$\alpha_5=\sum_{1\leq i<j\leq 4}  X_{ij}\varphi_i \wedge \varphi_j, \ \ \ \alpha_6= \sum_{1\leq i<j\leq 4} Y_{ij}\varphi_i \wedge \varphi_j,$$
where $X$ and $Y$ are smooth functions taking values in the skew-symmetric matrices, so that $\{ \alpha_5, \alpha_6\}$ forms a unitary frame of $W_1^{\perp}$. For $1\leq j\leq 4$, let $T^j=(T^j_{ik})$. Since $\{\alpha_1,\ldots,\alpha_6\}$ is a unitary frame of $W$, the $4\times 4$ matrix
\begin{equation} \label{eq:6sum}
 T^1(T^1)^{\ast} + T^2(T^2)^{\ast}+ T^3(T^3)^{\ast}+T^4(T^4)^{\ast}+XX^{\ast}+YY^{\ast}
 \end{equation}
is independent of the choice of unitary frames of $W$. If we choose the standard unitary frame $\{ \varphi_i\wedge \varphi_j\}_{1\leq i<j\leq 4}$ for $W$, we see that the matrix in (\ref{eq:6sum}) equals $3I_4$. Meanwhile, by the definition of the $A$ tensor, $A_{i\bar{j}} =\sum_{r,s=1}^4 T^r_{is} \overline{T^r_{js}}$, we see that the sum of the first four terms in (\ref{eq:6sum}) is equal to $A$, which is $2I_4$ since $S=0$ implies $A=B$. Therefore we have
\begin{equation} \label{eq:XY}
XX^{\ast}+YY^{\ast} = I_4,
\end{equation}
for any unitary frame $\{\alpha_5, \alpha_6\}$ of $W_1^{\perp}$ in the open neighborhood $U_0\subset M$. We now claim that there is a non-empty open subset $U\subset U_0$ on which smooth functions $a$ and $b$ can be chosen such that $|a|^2+|b|^2=1$ and $\tilde{\alpha}_5\wedge \tilde{\alpha}_5=0$ identically in $U$, where $\tilde{\alpha}_5 = a\alpha_5 +b\alpha_6$. To see this, note that
\begin{equation} \label{eq:q}
 \tilde{\alpha}_5\wedge \tilde{\alpha}_5 = (a^2 q_{X\!X} + 2ab \,q_{X\!Y} + b^2q_{YY})\,\varphi_1\wedge \varphi_2\wedge \varphi_3\wedge \varphi_4,
 \end{equation}
where $q_{X\!X}$ is the function defined by $\alpha_5\wedge \alpha_5 = q_{X\!X}\,\varphi_1\wedge \varphi_2\wedge \varphi_3\wedge \varphi_4$ and $q_{X\!Y}$, $q_{YY}$ are defined similarly. If $q_{X\!X}=0$ identically in $U_0$, then we may simply take $a=1$, $b=0$ in $U=U_0$, and the claim is satisfied. Assume that $q_{X\!X}$ is not identically zero. Let $U_1\subset U_0$ be a non-empty open subset in which $q_{X\!X}$ is nowhere zero. Consider the discriminant function $\Delta = (q_{X\!Y})^2- q_{X\!X} \,q_{YY}$. If $\Delta=0$ identically in $U_1$, then the quadratic polynomial on the right hand side of (\ref{eq:q}) is a perfect square $(af+hb)^2$ with $f$ being nowhere zero. In this case we may take $U=U_1$ and $a=tb$, $b=1/\sqrt{1+|t|^2}$, where $t=-\frac{h}{f}$, so the claim is satisfied. If $\Delta$ is not identically zero in $U_1$, take a non-empty open subset $U\subset U_1$ on which $\Delta$ is nowhere zero. On $U$, we can choose a smooth branch $\sqrt{\Delta}$ of the square root of $\Delta$ and set $a=tb$, $b=1/\sqrt{1+|t|^2}$, $t=(-q_{X\!Y}+\sqrt{\Delta})/q_{X\!X}$. This establishes the claim.

In summary, replacing $U_0$ by a smaller open subset $U$ if necessary, we can always find another unitary frame $\{ \tilde{\alpha}_5, \tilde{\alpha}_6\}$ of $W_1^{\perp}$ so that $ \tilde{\alpha}_5\wedge  \tilde{\alpha}_5=0$ identically. For simplicity, let us still denote them by $\alpha_5$ and $\alpha_6$, so now we have $\alpha_5\wedge  \alpha_5=0$, which means that the matrix $X$ has zero determinant everywhere in $U$.

Choose a new unitary frame $e$ in $U\subset M$ so that under $e$ the skew-symmetric matrix $X$ is block diagonal. Since its determinant is zero, it must be of the form
$$ X = \left[ \begin{array}{cc} \lambda E & 0 \\ 0 & 0 \end{array} \right], \ \ \  \ \mbox{where} \ \  E = \left[ \begin{array}{cc} 0 & 1 \\ -1 & 0 \end{array} \right] , $$
and $\lambda$ is a smooth function. Since $\frac{1}{2}\mbox{tr}(XX^{\ast})=|\alpha_5|^2 =1$, we see that $|\lambda|^2=1$.
This together with (\ref{eq:XY}) indicates that under the same $e$ we have
$$ Y = \left[ \begin{array}{cc} 0 & 0 \\ 0 & \mu E  \end{array} \right], $$
where $|\mu|^2 = 1$. Under this particular unitary frame $e$, for any $1\leq i\leq 4$, by the orthogonality condition $\langle T^i,\overline{X}\rangle = 0$ and  $\langle T^i,\overline{Y}\rangle = 0$, we get $T^i_{12}=T^i_{34}=0$ for any $1\leq i\leq 4$, so Claim A is proved.

\vspace{0.3cm}

\noindent {\bf Claim B:} There exists a non-empty open subset $U'\subset U$ and a unitary frame $e$ in $U'$ so that, up to skew-symmetry in the lower indices, the only non-zero torsion components under $e$ are
\begin{equation} \label{eq:T}
T^1_{23}=T^2_{14}=T^3_{24}=T^4_{13}=1.
\end{equation}

By Claim A, we have a non-empty open subset $U\subset M$ and a unitary frame $e$ in $U$ under which $T^{\ast}_{12}=T^{\ast}_{34}=0$. We will perform unitary changes on $\{ e_1, e_2\}$ and on $\{ e_3, e_4\}$ on a possibly smaller open subset $U'\subset U$ so that (\ref{eq:T}) holds. First let us write
\begin{eqnarray*}
&&  P^1=  \left[ \begin{array}{cc} T^1_{13} & T^1_{14} \\ T^1_{23} & T^1_{24} \end{array} \right]  = \left[ \begin{array}{c} ^t\!u \\ ^t\!v \end{array} \right] ,  \ \ \ \ \ \ P^2=  \left[ \begin{array}{cc} T^2_{13} & T^2_{14} \\ T^2_{23} & T^2_{24} \end{array} \right]  = \left[ \begin{array}{c} ^t\!w \\ -\,^t\!u \end{array} \right] ,\\
&& P^3=  \left[ \begin{array}{cc} T^3_{13} & T^3_{14} \\ T^3_{23} & T^3_{24} \end{array} \right]  = \ \left[  x ,y \right] ,  \ \ \ \ \ \ \  P^4=  \left[ \begin{array}{cc} T^4_{13} & T^4_{14} \\ T^4_{23} & T^4_{24} \end{array} \right]  = \ \left[  z , -x  \right] ,
\end{eqnarray*}
where $u$, $v$, $w$, $x$, $y$, and $z$ are smooth functions taking values in ${\mathbb C}^2$, written as column vectors. Let us also denote by $E_3=(T^j_{i3})$, $E_4=(T^j_{i4})$ the $2\times 2$ matrices where $1\leq i,j\leq 2$, and by $F_1=(T^{\beta}_{1\alpha})$, $F_2=(T^{\beta}_{2\alpha})$ the $2\times 2$ matrices where $3\leq \alpha , \beta \leq 4$, namely,
$$ E_3 = \left[ \begin{array}{cc} u_1 & w_1 \\ v_1 & -u_1 \end{array} \right], \ \ E_4 = \left[ \begin{array}{cc} u_2 & w_2 \\ v_2 & -u_2 \end{array} \right],\ \  F_1 = \left[ \begin{array}{cc} x_1 & z_1 \\ y_1 & -x_1 \end{array} \right], \ \ F_2 = \left[ \begin{array}{cc} x_2 & z_2 \\ y_2 & -x_2 \end{array} \right]. $$

\vspace{0.2cm}

We claim that at any point $p\in U$, the matrices $E_3$ and $E_4$ cannot be proportional to each other, and similarly, $F_1$ and $F_2$ cannot be proportional to each other.

To see this, assume on the contrary that $E_3$ is proportional to $E_4$ at $p$.  First of all, if $F_1$ is also proportional to $ F_2$ at $p$, then by performing a unitary change of $\{ e_3(p), e_4(p)\}$, we may assume that $E_4=0$. Similarly, by performing a unitary change of $\{ e_1(p), e_2(p)\}$, we may assume that $F_2=0$. That is, we may assume that $u_2=v_2=w_2=0$ and $x_2=y_2=z_2=0$, so we have
$$ P^1 = \left[ \begin{array}{cc} u_1 & 0 \\ v_1 & 0 \end{array} \right], \ \ P^2 = \left[ \begin{array}{cc} w_1 & 0 \\ -u_1 & 0 \end{array} \right],\ \  P^3 = \left[ \begin{array}{cc} x_1 & y_1 \\ 0 & 0 \end{array} \right], \ \ P^4 = \left[ \begin{array}{cc} z_1 & -x_1 \\ 0 & 0 \end{array} \right]. $$
Since $B_{4\bar{4}}>0$, we have $(z_1, x_1)\neq (0,0)$. By $B_{1\bar{3}}=B_{1\bar{4}}=0$, we get $u_1\overline{x}_1= u_1\overline{z}_1=0$, hence $u_1=0$. Similarly, by $B_{2\bar{3}}=B_{2\bar{4}}=0$ we get $w_1=0$. The equalities $u_1=w_1=0$ imply that $B_{2\bar{2}}=0$, a contradiction. So $F_1$ and $F_2$ cannot be proportional to each other at $p$.

As before, by a unitary rotation of $\{ e_3(p), e_4(p)\}$ we may assume that $E_4=0$, hence $u_2=v_2=w_2=0$. We may also rotate $\{ e_1(p), e_2(p)\}$ to make $x_2=0$, so we have
$$ P^1 = \left[ \begin{array}{cc} u_1 & 0 \\ v_1 & 0 \end{array} \right], \ \ P^2 = \left[ \begin{array}{cc} w_1 & 0 \\ -u_1 & 0 \end{array} \right],\ \  P^3 = \left[ \begin{array}{cc} x_1 & y_1 \\ 0 & y_2 \end{array} \right], \ \ P^4 = \left[ \begin{array}{cc} z_1 & -x_1 \\ z_2 & 0 \end{array} \right]. $$
If $x_1\neq 0$, then by $B_{1\bar{3}}=B_{2\bar{3}}=0$ we get $u_1=w_1=0$, which leads to $B_{2\bar{2}}=0$, a contradiction, so we must have $x_1=0$. Now by looking at $P^1$, $P^2$, and $P^4$, we see that their left columns are three non-zero column vectors in ${\mathbb C}^2$ that are mutually perpendicular to each other. This is impossible, proving the claim that at every point of $U$, $E_3$ is not proportional to $E_4$ and $F_1$ is not proportional to $ F_2$.

\vspace{0.2cm}

Next, let us perform a unitary change on $\{ e_3, e_4\}$ to assume that $\det(E_4)=0$. This may require us to shrink our open subset $U$ to ensure the smoothness of the new frame, but it can be achieved by finding smooth roots of quadratic equations in exactly the same way as in the proof of Claim A. Meanwhile, note that when we change the basis $\{ e_1, e_2\}$  by a unitary matrix $Z$, the matrix $E_4$  will be changed to $ZE_4Z^{\ast}$. So after the basis change in $\{ e_1, e_2\}$, we may assume that $E_4$ is upper triangular. Again here the change can be made smoothly by shrinking the open subset $U$ if needed.

Since  both the determinant and  the trace of $E_4$ are zero, the upper triangular matrix must be strictly upper triangular, hence $u_2=v_2=0$, and we have
$$ P^1 = \left[ \begin{array}{cc} u_1 & 0 \\ v_1 & 0 \end{array} \right], \ \ P^2 = \left[ \begin{array}{cc} w_1 & w_2 \\ -u_1 & 0 \end{array} \right],\ \  P^3 = \left[ \begin{array}{cc} x_1 & y_1 \\ x_2 & y_2 \end{array} \right], \ \ P^4 = \left[ \begin{array}{cc} z_1 & -x_1 \\ z_2 & -x_2 \end{array} \right]. $$
We also have $w_2\neq 0$, as otherwise $E_4=0$ which would violate the fact that $E_3$ and $E_4$ are not proportional. The left column of $P^1$ is a unit vector, and by $B_{1\bar{2}}=0$, we see that the left column of $P^2$ must be perpendicular to the left column of $P^1$, namely, there is a scalar valued function $a$ such that
$$ w_1= a\overline{v}_1, \ \ -u_1 =-a\overline{u}_1. $$
Since $B=2I$ and only mixed lower indices remain, $\|P^2\|^2=B_{2\bar{2}}/2=1$. Thus $|a|^2+|w_2|^2=1$, and $|a|^2<1$ because $w_2\neq 0$. By taking absolute values of both sides of the second equation in the above line, we conclude that $u_1=0$. Therefore $v_1\neq 0$, and by $B_{1\bar{3}}=B_{1\bar{4}}=0$, we know that $x_2=z_2=0$, so we have
$$ P^1 = \left[ \begin{array}{cc} 0 & 0 \\ v_1 & 0 \end{array} \right], \ \ P^2 = \left[ \begin{array}{cc} w_1 & w_2 \\ 0 & 0 \end{array} \right],\ \  P^3 = \left[ \begin{array}{cc} x_1 & y_1 \\ 0 & y_2 \end{array} \right], \ \ P^4 = \left[ \begin{array}{cc} z_1 & -x_1 \\ 0 & 0 \end{array} \right]. $$
The first rows of $P^2$ and $P^4$ are unit vectors perpendicular to each other. Now by $B_{3\bar{2}}=B_{3\bar{4}}=0$,  the first row of $P^3$ has to be perpendicular to both of them, so it must be zero, namely, we have $x_1=y_1=0$. Only the upper left entry of $P^4$ is non-zero. Since $B_{2\bar{4}}=0$, we also have $w_1=0$. Thus
$$ P^1 = \left[ \begin{array}{cc} 0 & 0 \\ v_1 & 0 \end{array} \right], \ \ P^2 = \left[ \begin{array}{cc} 0 & w_2 \\ 0 & 0 \end{array} \right],\ \  P^3 = \left[ \begin{array}{cc} 0 & 0 \\ 0 & y_2 \end{array} \right], \ \ P^4 = \left[ \begin{array}{cc} z_1 & 0 \\ 0 & 0 \end{array} \right]. $$
In other words, we may choose a unitary frame $e$ under which, up to skew-symmetry in the lower indices, the only non-zero torsion components are
$$ a_1:=T^1_{23}, \ \ \ a_2:=T^2_{14}, \ \ \ a_3:=T^3_{24}, \ \ \ a_4:=T^4_{13}, $$
with $|a_i|=1$ for each $i$. Choose smooth functions $\rho_i$ such that
$$  \rho_1^5=\overline{a}_1a_2a_3^2a_4^3, \ \ \ \ \ \rho_2=\overline{\rho}_1a_3a_4, \ \ \ \ \  \rho_3 = \rho_1^2 a_1\overline{a}_3\overline{a}_4, \ \ \ \ \ \rho_4 = \overline{\rho}_1^2a_2a_3a_4. $$
Then under the new unitary frame  $\{ \overline{\rho}_1e_1, \ldots , \overline{\rho}_4e_4\}$ all four of the above torsion components take the value $1$, namely, $T^1_{23}=T^2_{14}=T^3_{24}=T^4_{13}=1$, while all components other than these and their skew-symmetric counterparts are zero. So Claim B is proved.

\vspace{0.2cm}

Now we proceed with the proof of Proposition \ref{prop4c}. We will argue within the open subset where (\ref{eq:T}) holds. Since $T^j_{12}=0$ for any $j$, we have
$$ 0 = d T^j_{12} = \sum_{r=1}^4 \big\{ \theta^b_{1r} T^j_{r2} + \theta^b_{2r} T^j_{1r} \big\} = -\theta^b_{13} T^j_{23} -\theta^b_{14} T^j_{24} + \theta^b_{23} T^j_{13} +\theta^b_{24} T^j_{14}.$$
Taking $j=1,2,3,4$ in turn and using (\ref{eq:T}), we conclude that $\theta^b_{13}=\theta^b_{14}=\theta^b_{23}=\theta^b_{24}=0$. That is, if we write $P=\{ 1,2\}$ and $Q=\{3,4\}$, then we have $\theta^b_{i\alpha}=0$ for any $i\in P$ and for any $\alpha \in Q$. We could pursue this line further to eventually conclude that $R^b=0$, which will give us a contradiction as non-K\"ahler Bismut flat metrics cannot be balanced (see, for instance, \cite{WYZ}). Alternatively, since we have the constant Chern holomorphic sectional curvature assumption here, we could use (\ref{eq:Rb}) to draw a quicker contradiction. For any $i\in P$ and $\alpha \in Q$, the vanishing of $\theta^b_{i\alpha}$ implies that $\Theta^b_{i\alpha}=0$. Since $c=0$ and $T^p_{pq}=0$ for any $p$, $q$, by (\ref{eq:Rb}) we have
\begin{eqnarray*}
 0& = &  R^b_{i\bar{\alpha}\alpha \bar{i}} \ \ = \ \ \sum_r \big\{ \frac{1}{2} |T^r_{i\alpha}|^2 +\frac{1}{4}T^i_{ir} \overline{T^{\alpha}_{\alpha r} } +\frac{1}{4} T^{\alpha}_{\alpha r} \overline{ T^i_{ir}}  -\frac{3}{4} |T^i_{\alpha r}|^2 -  \frac{3}{4} |T^{\alpha}_{ir}|^2   \big\} \\
 & = & \frac{1}{2} \sum_{j\in P} |T^j_{i\alpha}|^2 + \frac{1}{2} \sum_{\beta \in Q} |T^{\beta}_{i\alpha}|^2 - \frac{3}{4}\sum_{j\in P} |T^i_{j \alpha }|^2 - \frac{3}{4} \sum_{\beta \in Q} |T^{\alpha}_{i\beta }|^2.
 \end{eqnarray*}
If we sum over $i\in P$ and $\alpha\in Q$, then we end up with
$$ 0 \ \, = \ \,-\frac{1}{4} \sum_{i,j\in P} \sum_{\alpha \in Q} |T^j_{i\alpha}|^2 -\frac{1}{4} \sum_{i\in P} \sum_{\alpha ,\beta \in Q} |T^{\alpha}_{i\beta }|^2. $$
Therefore $T^{\ast}_{i\alpha}=0$ for any $i\in P$ and any $\alpha \in Q$, hence $T=0$, a contradiction. This establishes Proposition \ref{prop4c}, which also completes the proof of Proposition \ref{prop4b}.
\end{proof}

Putting Propositions \ref{prop4a} and \ref{prop4b} together, we know that for any balanced BTP manifold of dimension $4$ with constant Chern holomorphic sectional curvature, the rank $r_B$ of the $B$ tensor must be less than $4$, and by Propositions \ref{proprank1} and \ref{proprank2}, $r_B$ cannot be $1$ or $2$, either. So when the metric is non-K\"ahler, the only possibility is $r_B=3$. In this case, under the compactness assumption, Proposition \ref{proprank3} says that the manifold $(M^4,g)$ has to be a compact quotient of the reductive complex Lie group $\mbox{SL}(2,{\mathbb C})\times {\mathbb C}$, equipped with (a constant multiple of) the product of the standard Killing metric with the flat Euclidean metric. This completes the proof of Theorem \ref{thm1}.

\vspace{0.3cm}

\noindent\textbf{Acknowledgments.}
The second author would like to thank Haojie Chen, Shuwen Chen, Xiaolan Nie, Kai Tang, Bo Yang, and Xiaokui Yang for their interest and helpful discussions. He is very grateful to Quanting Zhao for their long collaboration on BTP manifolds and numerous discussions.

\vspace{0.3cm}

\noindent\textbf{Declaration on generative AI use.}
At an earlier stage of this study, a suggestion generated by GPT-5.5 Pro \cite{OpenAIGPT55Pro} led us to test the corrected algebraic reductions by a brute-force Gr\"obner-basis calculation for the case of balanced BTP fourfolds with non-zero constant Chern holomorphic sectional curvature. We subsequently carried out this computation; it gave a computer-assisted verification of the K\"ahlerness in this subcase. The present article uses a different approach and does not rely on computer-assisted computations.

\vspace{0.3cm}

\noindent\textbf{Author contributions.}
All authors declare that each of them made substantial contributions to the conception of the work, drafted the work or revised it critically for important intellectual content, approved the version to be published, and agrees to be accountable for all aspects of the work in ensuring that questions related to the accuracy or integrity of any part of the work are appropriately investigated and resolved.

\vspace{0.3cm}

\noindent\textbf{Competing interests.}
All authors declare that there are no competing interests for this article.

\end{document}